\documentclass[11pt, reqno]{amsart}

\usepackage{amssymb, amsthm, amsmath}
\usepackage{comment}
\usepackage{graphics}
\usepackage{hyperref}
\usepackage{xcolor}
\usepackage[all]{xy}
\usepackage[mathscr]{eucal}
\usepackage[shortlabels]{enumitem}
\usepackage{thmtools}
\usepackage{amsrefs}
\usepackage{setspace}
\usepackage{mathtools}
\usepackage{stmaryrd}
\usepackage{upgreek}
\usepackage{tikz}
\usepackage{mathpazo}

\tikzset{
v/.style={draw, fill, circle, minimum size=1.5mm, inner sep=0},
b/.style={draw , circle, minimum size=2.5mm, inner sep=.5mm},
e/.style={very thick},
vs/.style={draw, fill, circle, minimum size=1mm, inner sep=0},
bs/.style={draw, fill, circle, minimum size=1.5mm, inner sep=0mm},
es/.style={thick}
}
\newlength{\nodeheight}
\newlength{\nodewidth}
\usetikzlibrary{arrows,matrix,positioning}
\usetikzlibrary {shapes.geometric} 

\numberwithin{equation}{section}

\declaretheorem[name=Theorem, sibling=equation]{theorem}
\declaretheorem[name=Lemma, sibling=equation]{lemma}
\declaretheorem[name=Proposition, sibling=equation]{proposition}

\declaretheorem[name=Definition, style=definition, sibling=equation]{definition}

\DeclareMathOperator{\Hom}{Hom}

\DeclareMathOperator{\End}{End}

\newcommand{\C}{{\mathscr{C}}}

\newcommand{\N}{{\mathbb{N}}}
\newcommand{\R}{{\mathscr{R}}}
\newcommand{\ov}[1]{\overline{#1}}
\newcommand{\inj}{\hookrightarrow}
\newcommand{\sur}{\twoheadrightarrow}
\newcommand{\FI}{{\rm FI}}
\newcommand{\RBr}{\R{\rm Br}}
\newcommand{\Br}{{\rm Br}}

\title{Noetherianity of Diagram Algebras}

\author{Anthony Muljat and Khoa Ta}
\address{}
\email{}

\begin{document}

\begin{abstract}
In this short paper, we establish the local Noetherian property for the linear categories of Brauer, partition algebras, and other related categories of diagram algebras with no restrictions on their various parameters.
\end{abstract}

\maketitle

\section{Introduction}
In this paper, we prove the local Noetherian property for the linear categories of Brauer and partition algebras as introduced by Patzt in \cite{Patzt}. Additionally, we extend our results to two related categories of diagram algebras: the Rook-Brauer and Rook algebras.  Our results are valid without any restrictions on the various parameters of these algebras.

For a commutative Noetherian ring $R$, the {\it partition algebra} $P_n = P_n(R,\delta)$ with parameter $\delta \in R$ is a free $R$-module over the basis of all partitions of the union of the sets $[-n] = \{-n,\ldots,-1\}$ and $[n] = \{1,\ldots,n\}$. Each basis element of $P_n$ may be visualized by a diagram $\alpha$ with two vertical columns of $n$ nodes each, with the nodes on the left representing $-n$ through $-1$ and the nodes on the right representing $1$ through $n$, with paths connecting certain nodes. The connected components of $\alpha$ indicate the blocks of the corresponding partition; we identify any diagrams whose components contain the same nodes. 

\begin{figure}[h!]
\centering
\includegraphics[scale=0.7]{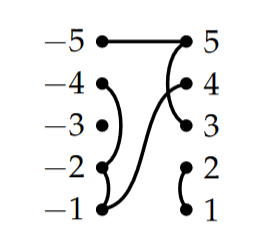}
\caption{Visualization of the partition $\{\{-5,5,3\},\{-4,-2,-1,4\},\{-3\},\{1,2\}\}$.}
\label{fig:P5ex}
\end{figure}

For partitions $p$ and $q$, we define the product $pq$ in the following way. First we conjoin the diagrams $p$ and $q$ by identifying each node $k$ on the right of $p$ with the node $-k$ on the left of $q$, thus forming a diagram with $3$ columns of $n$ nodes each. We then form a new partition $s$ of $[-n] \cup [n]$ in which two elements belong to the same block if and only if the corresponding nodes in the left or right column of the conjoined diagram are connected. We define $pq = \delta^r \cdot s$ where $r$ is the number of connected components in the conjoined diagram that only contain nodes in the middle column.

As an example, below we have a multiplication of two basis diagrams in $P_5(R,\delta)$. Note that there is a factor of $\delta$ in the result to indicate a single connected component (shown in red) that only contains nodes in the middle column.

\begin{figure}[h!]
\centering
\includegraphics[scale=0.7]{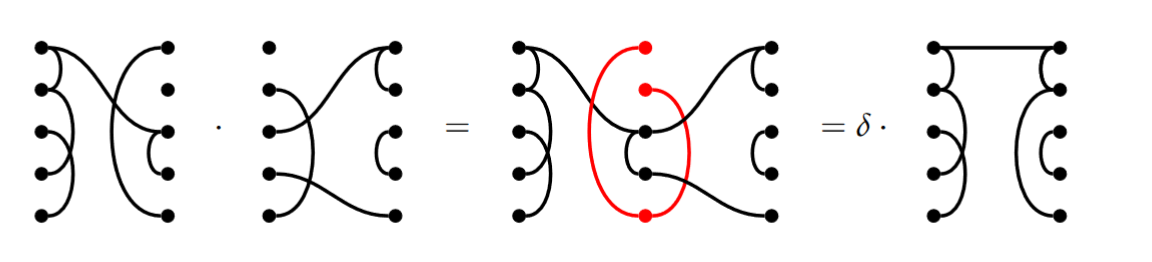}
\caption{Multiplication in $P_5(R,\delta)$.}
\label{fig:MultP5}
\end{figure}

The {\it Brauer algebra} $\Br_n = \Br_n(R,\delta)$ is the free subalgebra of $P_n(R,\delta)$ generated by partitions that are perfect matchings; that is, the blocks in each partition have size 2. Thus, each basis element is represented by a diagram in which each node is connected to exactly one other node.

Two other diagram algebras related to the partition algebra are the Rook-Brauer algebra and the Rook algebra. For a commutative Noetherian ring $R$, the {\it Rook-Brauer algebra} $\RBr_n = \RBr_n(R, \delta,\epsilon)$ with parameters $\delta, \epsilon \in R$ is a free $R$-module with basis consisting of partitions of $[-n] \cup [n]$ whose blocks have size $\le 2$. Thus, each basis element is represented by a diagram in which each node is connected to at most one other node. The product of two basis diagrams $p,q \in \RBr_n$ is defined by $pq = \delta^c \epsilon^d s$ where $s$ is the partition obtained as described above, with $c$ factors of $\delta$ for $c$ loops and $d$ factors of $\epsilon$ for $d$ contractible components in the conjoined diagram that only contain nodes in the middle column (see example below).

\begin{figure}[h!]
\centering
\includegraphics[scale=0.7]{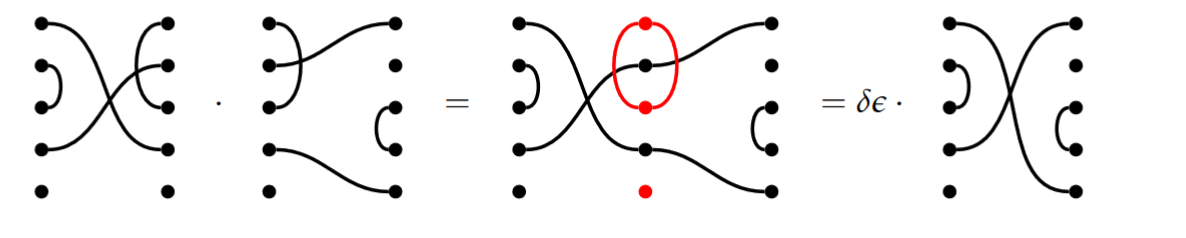}
\caption{Multiplication in $\RBr_5(\delta,\epsilon)$.}
\label{fig:MultP5}
\end{figure}

\noindent Here, the final diagram is multiplied by a factor of $\delta$ for the red loop and a factor of $\epsilon$ for the isolated node in the middle.

The {\it Rook algebra} $\R_n = \R_n(R, \epsilon)$ is the free subalgebra of $\RBr_n$ generated by diagrams without any left-to-left or right-to-right edges (i.e. the {\it invertible} diagrams of $\RBr_n$). Note that the parameter $\delta$ can be removed as multiplication of diagrams with no left-to-left or right-to-right edges cannot yield any loops in the middle. 

\section{Main Result}
\noindent We now introduce the background necessary to discuss our main results. In the following, we fix a commutative Noetherian ring $R$ and parameters $\delta, \epsilon \in R$.

Following Patzt's \cite{Patzt} convention, given a nonnegative integer $n \in \N$, we will denote by $A_n$ any one of $P_n$, $\Br_n$, $\RBr_n$, or $\R_n$. For each $A_n$, there is a canonical choice of a ``trivial'' $A_n$-module $R$ on which all invertible diagrams act trivially and which is annihilated by all other diagrams. We denote by $A_{n-m}$ the subalgebra of $A_n$ generated by diagrams with horizontal connections $\{-k,k\}$ for all $1 \le k \le m$.

Denote by $\{A_n\}$ any one of the sequences $\{P_n\}$, $\{\Br_n\}$, $\{\RBr_n\}$, or $\{\R_n\}$. We define a category $\C_A$ enriched over the category of $R$-modules by the following data:
\begin{itemize}
    \item The objects of $\C_A$ are the nonnegative integers $\N$.
    \item If $m \le n$, we define $\Hom_{\C_A}(m,n) := A_n \otimes_{A_{n-m}} R$; otherwise $\Hom_{\C_A}(m,n) := 0$ is the trivial $R$-module. Note that $\End_{\C_A}(n) \cong A_n$.
    \item By \cite{Patzt}*{Thm. 2.1}, there is a surjection
    \[
        \Hom_{\C_A}(m,n) \otimes \Hom_{\C_A}(l,m) \to \Hom_{\C_A}(l,n) \qquad (l \le m \le n)
    \]
    which provides the composition $l \to m \to n$ for $l \le m \le n$; otherwise the composition is $0$.
\end{itemize}

We thereby obtain categories $\C_{P}$, $\C_{\Br}$, $\C_{\RBr}$, and $\C_{\R}$ enriched over the category of $R$-modules corresponding to the sequences $\{P_n\}$, $\{\Br_n\}$, $\{\RBr_n\}$, and $\{\R_n\}$ respectively.

\begin{definition}
    A \emph{$\C_A$-module} $V$ is a functor from the category $\C_A$ to the category of $R$-modules. For each $n \in \N$, we denote the $R$-module $V(n)$ by $V_n$.
\end{definition}

A \emph{$\C_A$-module map} is a natural transformation of $\C_A$-modules. Given $\C_A$-modules $V$ to $W$, if $W_n \subseteq V_n$ for each $n$ and the map $W \to V$ induced by the inclusions $W_n \hookrightarrow V_n$ is a $\C_A$-module map, then $W$ is a \emph{$\C_A$-submodule} of $V$. We say that $V$ is \emph{finitely generated} if there exists a finite set $S\subseteq \underset{n}{\sqcup}~ V_n$ such that the smallest $\C_A$-submodule of $V$ containing $S$ is $V$ itself.  

\begin{definition}
    A $\C_A$-module $V$ is \emph{Noetherian} if every $\C_A$-submodule of $V$ is finitely generated.  We say $\C_A$ is \emph{locally Noetherian} over $R$ if every finitely generated $\C_A$-module is Noetherian.
\end{definition}

We can now state our main result. 

\begin{theorem}
    For $A = P$, $\Br$, $\RBr$, or $\R$, the category $\C_A$ is locally Noetherian.
\end{theorem}

The local Noetherian property for the categories of partition, Brauer, and Temperley-Lieb algebras, with the parameter $\delta$ chosen such that each algebra is semi-simple, was previously established in \cite{Gan-Ta}. Our result extends this by removing restrictions on various parameters, thereby providing a more general statement.

The crux of the main argument is to establish the existence of a functor from the category $\FI$ to the category $\C_A$ when $A = P$, $\RBr$, or $\R$. This implies that any $\C_A$-module, when pulled back, also becomes an $\FI$-module. Given that the local Noetherianity property is already known for the category $\FI$ \cite{cefn}*{Thm. A}, this result consequently extends the local Noetherianity property to the category $\C_A$. Unfortunately, this argument does not hold for the Temperley-Lieb algebra, as there is no nontrivial functor from the category $\FI$ to the category of Temperley-Lieb algebras. Therefore, we leave the reader with an open question.

\vskip1em
\noindent\textit{{\bf Question:} For any parameter $\delta$, is the category of Temperley-Lieb algebras $\C_{TL(\delta)}$ locally Noetherian?}  

\noindent {\bf Acknowledgement}: We would like to thank our advisor, Dr. Wee Liang Gan for his
useful suggestions and helpful discussion.
\section{Proof of Main Result}
It is known \cite{Patzt}*{Prop. 2.3} that $\Hom_{\C_{\Br}}(m,n)$ is a free $R$-module with a basis which may be represented by diagrams with $n$ nodes on the left, $m$ nodes on the right and an ``$(n-m)$-blob'' which is connected to exactly $n-m$ nodes, such that every node is connected to exactly one other node or the blob.  Similarly, by \cite{Patzt}*{Prop. 2.5}, $\Hom_{\C_P}(m,n)$ is a free $R$-module with basis given by the set of all partitions of $[-n] \cup [m]$ with $n-m$ ``marked'' blocks.

We now prove an analogous result for categories of Rook and Rook-Brauer algebras.

\begin{proposition} \label{RBrprop}
    $\Hom_{\C_{\RBr}}(m,n)$ is a free $R$-module with a basis consisting of diagrams with $n$ nodes on the left, $m$ nodes on the right and an ``$(n-m)$-blob'' which is connected to exactly $n-m$ nodes, such that every node is connected to at most one other node or the blob.
\end{proposition}

\begin{proof}
    Let $J \subset \RBr_n$ be the free $R$-submodule generated by diagrams which have either an isolated node $k$ with $m+1 \le k \le n$ or a right-to-right connection $\{i,j\}$ with $m+1 \leq i, j \le n$. Note that $J$ is closed under right multiplication by $\RBr_{n-m}$.
    
    We claim that the image of the map $J \otimes_{\RBr_{n-m}} R \to \RBr_n \otimes_{\RBr_{n-m}} R$ induced by the inclusion $J \inj \RBr_n$ is $0$. To prove this, we will show that every basis diagram $\alpha \in J$ may be written $\alpha = \beta \gamma$, where $\beta$ and $\gamma$ are basis diagrams of $\RBr_n$ and $\RBr_{n-m}$ respectively and $\gamma$ is non-invertible. Thus we will have
    \[
        \alpha \otimes 1 \mapsto \beta\gamma \otimes 1 = \beta \otimes \gamma \cdot 1 = \beta \otimes 0 = 0.
    \]
    
    {\it Case 1.} Assume $\alpha \in J$ has an isolated node $k$ between $m+1$ and $n$. Since $\alpha$ must have an even number of isolated nodes, we can find  another isolated node either $k'$ on the left or right of $\alpha$. Let $\beta \in \RBr_n$ be the diagram with the connection $\{k, k'\}$ and all other connections are the same as in $\alpha$. Let $\gamma \in \RBr_{n-m}$ be the diagram with a pair of isolated nodes on opposite sides at $\{k\}$, $\{-k\}$ with all other connections left-to-right $\{-l,l\}$. Then $\gamma$ is non-invertible and $\alpha = \beta \gamma$.

    {\it Case 2.} Now assume $\alpha$ has no isolated nodes between $m+1$ and $n$. This implies that $\alpha$ has a right-to-right connection $\{i,j\}$ where $m+1 \leq i,j \leq n$. It follows that $\alpha$ has either a pair of isolated nodes on the left hand side or a left-to-left connection. Denote this pair or connection by $(i',j')$, let $\beta \in \RBr_n$ be the diagram with connections $\{i, i'\}$ and $\{j, j'\}$ with all other connections the same as in $\alpha$. If $i'$ and $j'$ are left-to-left connected in $\alpha$, let $\gamma \in \RBr_{n-m}$ be the diagram with connections $\{-i,-j\}$, $\{i,j\}$, and all other connections left-to-right $\{-l, l\}$. If $i'$ and $j'$ are isolated nodes in $\alpha$, let $\gamma \in \RBr_{n-m}$ be the diagram with isolated nodes at $\{-i\}$, $\{-j\}$ and the connection $\{i,j\}$ with all other connections left-to-right $\{-l, l\}$. In either case $\gamma$ is non-invertible and $\alpha = \beta \gamma$. This completes the proof of the claim.
    
    The inclusion $J \inj \RBr_n$ and projection $\RBr_n \sur \RBr_n / J$ fit into a short exact sequence
    \[
        0 \to J \inj \RBr_n \sur \RBr_n / J \to 0.
    \]
    Applying $- \otimes_{\RBr_{n-m}} R$ yields the exact sequence
    \[
        J \otimes_{\RBr_{n-m}} R \to \RBr_n \otimes_{\RBr_{n-m}} R \to \RBr_n / J \otimes_{\RBr_{n-m}} R \to 0.
    \]
    Since $J \otimes_{\RBr_{n-m}} R \to \RBr_n \otimes_{\RBr_{n-m}}$ is $0$, it follows that 
    \begin{equation}\label{eq1}
        \RBr_n \otimes_{\RBr_{n-m}} R \cong \RBr_n / J \otimes_{\RBr_{n-m}} R.
    \end{equation}
    Note that the images in $\RBr_n / J$ of diagrams in $\RBr_n$ such that every node $i$ with $m \le i \le n$ is connected to exactly one other node $j$ with $j \le m$ form a basis of $\RBr_n / J$.
    
    Now let $I \subset \RBr_{n-m}$ be the free $R$-submodule generated by all noninvertible diagrams in $\RBr_{n-m}$. Then $I$ annihilates $R$ and $\RBr_n / J$. It follows that the action of $\RBr_{n-m}$ on $R$ and $\RBr_{n-m} / J$ factors through $\RBr_{n-m} / I \cong RS_{n-m}$, where $S_{n-m}$ is the subgroup of $S_n$ which acts trivially on $\{1, \ldots, m\}$ and acts by permutation on $\{m+1, \ldots, n\}$. Hence
    \begin{equation}\label{eq2}
        \RBr_n / J \otimes_{\RBr_{n-m}} R \cong \RBr_n / J \otimes_{RS_{n-m}} R.
    \end{equation}
    Together, (\ref{eq1}) and (\ref{eq2}) imply
    \begin{equation}\label{eq3}
        \RBr_n \otimes_{\RBr_{n-m}} R \cong \RBr_n / J \otimes_{RS_{n-m}} R.
    \end{equation}
    Since $S_{n-m}$ permutes the basis of $\RBr_n / J$, it follows that $\RBr_n / J \otimes_{RS_{n-m}} R$ is a free $R$-module with the desired basis. The result then follows from (\ref{eq3}).
\end{proof}

The analogous result for $\Hom_{\C_{\R}}(m,n)$ may be obtained by restricting the diagrams to those without left-to-left or right-to-right connections. We adjust the proof slightly to account for this restriction.

\begin{proposition} \label{Rprop}
    $\Hom_{\C_{\R}}(m,n)$ is a free $R$-module with a basis consisting of diagrams with $n$ nodes on the left, $m$ nodes on the right and an ``$(n-m)$-blob'' which is connected to exactly $n-m$ nodes, such that every node is connected to at most one other node or the blob, and all connections are left-to-right. 
\end{proposition}

\begin{proof}
    Let $J \subset \R_n$ be the free $R$-submodule generated by diagrams which have at least one isolated node $k$ with $m+1 \le k \le n$. As in Proposition \ref{RBrprop}, we claim that the image of $J \otimes_{\R_{n-m}} R \to \R_n \otimes_{\R_{n-m}} R$ is $0$.  Note that diagrams in Rook algebras must have the same number of isolated nodes on the left as on the right. Let $\alpha$ be any diagram in $J$. To each isolated node $k$ between $m+1$ and $n$, we may associate a unique isolated node $k' < 0$. Let $\beta \in \R_n$ be the diagram where the pairs $(k,k')$ are connected and all other nodes are connected as they are in $\alpha$. Let $\gamma \in \R_{n-m}$ be the diagram where all the $k$ and $-k$ nodes are isolated and all other nodes are connected horizontally. Then $\gamma$ is noninvertible and $\alpha = \beta\gamma$. The rest of the proof is the same, mutatis mutandis, as the proof of Proposition \ref{RBrprop}.
\end{proof}

Recall the category $\FI$ whose objects are the nonnegative integers and $\Hom_\FI(m,n)$ consists of injective maps $[m] \to [n]$ where $[m] = \{1, \ldots, m\}$; note that when $m = n$, we have $\Hom_\FI(m,n) = S_n$, the symmetric group on $n$ objects. An {\it $\FI$-module} is a functor from the category $\FI$ to the category of $R$-modules.

For $A = \Br$, $\RBr$, or $\R$, define a functor $F: \FI \to \C_{A}$ by the following data:
\begin{itemize}
    \item $F(n) = n$ for all $n \in \mathbb{Z}_{\ge 0}$.
    \item For all $\alpha \in \Hom_\FI(m,n)$, let $F(\alpha) \in \Hom_{\C_{A}}(m,n)$ be the basis diagram with connections $\{ -\alpha(i), i\}$ for all $i \in [m]$, and all $n-m$ remaining nodes on the left are connected to the $(n-m)$-blob.
\end{itemize}
\begin{figure}[h!]
\centering
\includegraphics[scale=0.7]{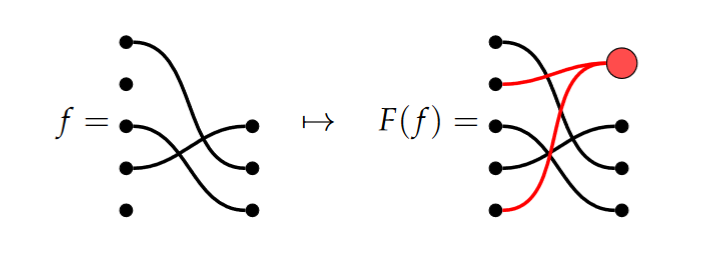}
\caption{Example of $F(f) \in \Hom_{\C_A}(3,5)$ with an injection $f:[3] \to [5]$} 
\end{figure}
Note that in the above example, the 2 nodes on the left of $F(f)$ not in the image of $f$ are connected to the $(5-3)$-blob on the right.

When $A = P$, define the functor $G: \FI \to \C_P$ by the following data:
\begin{itemize}
    \item $G([n]) =[n] $ for all $n \in \mathbb{Z}_{\ge 0}$. 
    \item For all $\alpha \in \Hom_\FI(m,n)$, let $F(\alpha) \in \Hom_{\C_{A}}(m,n)$ be the basis diagram with connections $\{ -\alpha(i), i\}$ for each $i \in [m]$ and all the remaining $n-m$ nodes on the left constitute the $n-m$ marked blocks.
\end{itemize}
\begin{figure}[h!]
\centering
\includegraphics[scale=0.7]{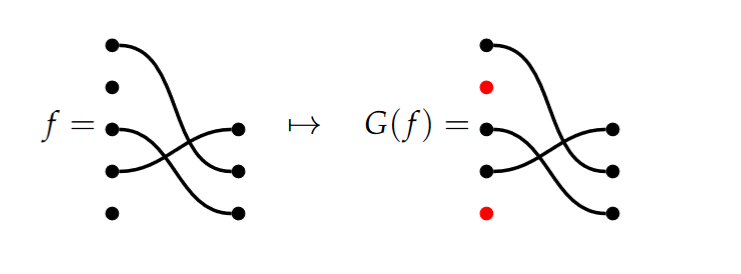}
\caption{Example of $G(f) \in \Hom_{\C_P}(3,5)$ with an injection $f:[3] \to [5]$} 
\end{figure}
Note that in the above example, $G(f)$ has $5-3=2$ marked blocks (colored red) in the form of two isolated nodes. 

With these functors, any $\C_{A}$-module may be considered as an $\FI$-module by pulling back along either $F$ or $G$.

For all $m \in \N$, we define the {\it free $\C_A$-module} $M_A(m)$ by the following data:
\begin{itemize}
    \item For all $n \in \N$, $M_A(m)_n$ is the free $R$-module generated on $\Hom_{\C_A}(m,n)$.
    \item The action of $\C_A$-morphisms is given by postcomposition; that is, for all $\alpha \in \Hom_{\C_A}(n,p)$, the induced map $M_A(m)(\alpha) : M_A(m)_n \to M_A(m)_p$ is defined by $\beta \mapsto \alpha \beta$ for all $\beta \in \Hom_{\C_A}(m,n)$.
\end{itemize}
Our objective now is to demonstrate that the free $\C_A$-module is finitely generated as an $\FI$-module. We start with the case when $A = P$, for which we establish a quick lemma.

\begin{lemma} \label{Lemma.for.Partition}
    If $j \geq 5m$, then any basis diagram $\alpha \in M_{\C_P}(m)_j$ has at least $m+1$ marked blocks consisting of a single node.
\end{lemma}

\begin{proof}
    Let $b$ be the number of marked blocks in $\alpha$ with a single node. For contradiction, assume $b \le m$. From \cite{Patzt}*{Prop. 2.5}, $\alpha$ contains $j-m$ marked blocks. So there are another $j-m-b$ marked blocks in $\alpha$, each of which has at least two nodes. Together, these blocks must contain at least $2(j-m-b)$ nodes, out of the $j+m-b$ nodes not already in a block. However, since $j \ge 5m$ and $0 \le b \le m$, it follows that $j > 3m+b$, or equivalently, $2(j-m-b) > j+m-b$. So there are not enough nodes to fill all the blocks, a contradiction.
\end{proof}

\begin{proposition} \label{Partition.fg}
    For $m \geq 0$, $M_{\C_P}(m)$ is finitely generated as an $\FI$-module. 
\end{proposition}
\begin{proof}
    The result is immediate for $m = 0$; suppose $m \geq 1$. We claim that $M_{\C_P}(m)$ is generated as an $\FI$-module by basis diagrams of $M_{\C_P}(m)_k$ where $0 \leq k \leq 5m$.
    
    First, we show that $M_{\C_P}(m)_{5m+1}$ is generated by basis diagrams of $M_{\C_P}(m)_{5m}$.  By Lemma \ref{Lemma.for.Partition}, any basis diagram $\alpha \in M_{\C_P}(m)_{5m+1}$ has at least $m+1$ marked blocks of size 1. Since $\alpha$ has only $m$ nodes on the right, there must be a marked block of size 1 on the left of $\alpha$. 
    
    We can find $\sigma \in S_{5m+1}$ such that $\sigma \alpha$ has the singe-node marked block on the top left.  Let $\iota \in \Hom_{\FI}(5m,5m+1)$ be the inclusion map $[5m] \hookrightarrow [5m+1]$ and $\ov{\alpha}$ be the diagram obtained from $\sigma \alpha$ by removing the marked block on the top left (see below).  Note that $\ov{\alpha} \in M_{\C_P}(m)_{5m}$ and $\sigma \alpha = G(\iota) \ov{\alpha}$, or equivalently $\alpha = \sigma^{-1} G(\iota) \ov{\alpha}$.   
    \begin{figure}[h!]
    \centering
    \includegraphics[scale=0.7]{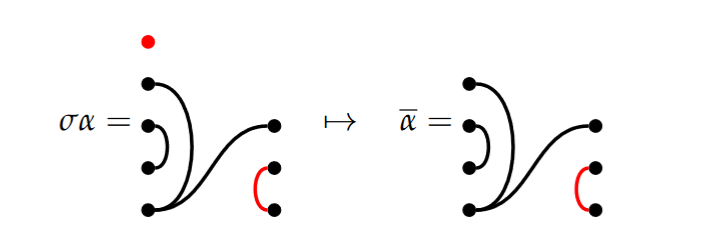}

    \caption{$\sigma\alpha$ to $\ov{\alpha}$} 
    \end{figure}

    \noindent Since $\sigma^{-1} \in S_{5m+1}$, we have $G(\sigma^{-1}) = \sigma^{-1}$ and $\alpha = G(\sigma^{-1} \iota) \ov{\alpha}$ with $\sigma^{-1} \iota \in \Hom_{\FI}(5m,5m+1)$. This implies that $M_{\C_P}(m)_{5m+1}$ is generated by basis diagrams of $M_{\C_P}(m)_{5m}$.
    
    A routine induction now shows that $M_{\C_P}(m)$ is generated as an $\FI$-module by basis diagrams of $M_{\C_P}(m)_k$ with $0 \leq k \leq 5m$.  Hence, $M_{\C_P}(m)$ is finitely generated as an $\FI$-module. 
\end{proof}

We proceed to prove a similar result for other categories. 

\begin{lemma} \label{Lemma.for.Other}
    For $A = \Br$, $\R$, or $\RBr$, if $j > 2m$, any morphism of $M_{\C_A}(m)_j$ has at least a connection from the blob to some node on the left.
\end{lemma}

\begin{proof}
    By \cite{Patzt}*{Prop. 2.3} and Propositions \ref{RBrprop} and \ref{Rprop} above, any basis diagram of $M_{\C_A}(m)_j$ must have exactly $j-m$ connections to the blob.  Since $j-m > m$ and we only have $m$ nodes on the right, there has to be at least one connection from the node on the left to the blob.
\end{proof}

\begin{proposition}\label{RBrfgprop}
    For $A = \Br$, $\RBr$, or $\R$, the free $\C_A$-module $M_{\C_{A}}(m)$ is finitely generated as an $\FI$-module for all $m \ge 0$.
\end{proposition}

\begin{proof}
    The case $m = 0$ is straightforward; assume $m \ge 1$. We claim that $M_{\C_{A}}(m)_{2m+1}$ is generated as an $\FI$-module by basis diagrams of $\Hom_{\C_{A}}(m,2m)$. Let $\alpha$ be a basis diagram of $\Hom_{\C_{A}}(m,2m+1)$. By Lemma \ref{Lemma.for.Other}, there is a connection from the blob to some node on the left of $\alpha$. Furthermore, we can find a $\sigma \in S_{2m+1}$ such that the top left node of $\sigma \alpha$ is connected to the blob. Let $\widetilde{\alpha} \in \Hom_{\C_{A}}(m,2m)$ be the diagram obtained by deleting the top row of $\sigma \alpha$, and let $\iota \in \Hom_{\FI}(2m,2m+1)$ be the inclusion map. We have $\sigma \alpha = F(\iota) \widetilde{\alpha}$, so that $\alpha = \sigma^{-1}F(\iota) \widetilde{\alpha} = F(\sigma^{-1} \iota) \widetilde{\alpha}$. This establishes the claim.

    A standard induction argument then demonstrates that $M_{\C_{A}}(m)$ is generated as an $\FI$-module by basis diagrams of $\Hom_{\C_{A}}(m,n)$ where $n \le 2m$.
\end{proof}

Finally, we can prove the main result. 

\begin{theorem}
    For $A = P$, $\Br$, $\RBr$, or $\R$, any finitely generated $\C_{A}$-module is Noetherian.
\end{theorem}

\begin{proof}
    Since any finitely generated $\C_A$-module $V$ admits a surjection $\bigoplus M_{\C_A}(m_i) \to V$ for some finite sequence of nonnegative integers $\{m_i\}$, it suffices to show that $M_{\C_A}(m)$ is a Noetherian $\C_A$-module for all $m \ge 0$. Let $W$ be a $\C_A$-submodule of $M_{\C_A}(m)$, then $W$ is also an $\FI$-submodule.  Propositions \ref{Partition.fg} and \ref{RBrfgprop}, combined with the local Noetherian property for $\FI$ \cite{cefn}*{Thm. A}, establish that $W$ is finitely generated as an $\FI$-submodule. Consequently, $W$ is also finitely generated as a $\C_A$-submodule, thereby demonstrating that $M_{\C_A}(m)$ is a Noetherian $\C_A$-module.
\end{proof}

\end{document}